\documentclass[a4]{article}
\usepackage{hyperref}
\usepackage{amsmath}
\usepackage{amssymb}
\usepackage{amsthm}
\usepackage{graphicx}
\usepackage{graphics}
\usepackage{tikz-cd}
\usepackage{cite}
\usepackage[mathscr]{euscript}
\usepackage{shuffle}
\usepackage{enumitem}
\theoremstyle{plain}
\newtheorem{Th}{Theorem}[section]
\newtheorem{Lemma}[Th]{Lemma}

\newtheorem{Cor}[Th]{Corollary}
\newtheorem{Prop}[Th]{Proposition}

\theoremstyle{definition}

\newtheorem{Rem}[Th]{Remark}
\newtheorem{?}[Th]{Problem}
\newtheorem{Ex}[Th]{Example}

\newcommand{\Z}{\mathbb{Z}}

\DeclareMathOperator{\id}{id}

\DeclareSymbolFont{rsfs}{U}{rsfs}{m}{n}
\DeclareSymbolFontAlphabet{\mathscrsfs}{rsfs}

\newcommand{\bw}{{\textstyle\bigwedge}}
\newcommand{\ad}{\mathrm{ad}}

\newcommand{\uf}{U\mathfrak{f}_3}
\newcommand{\lamlam}{\lambda/[\lambda,\lambda]}

\begin{document}
\title{The kernel of formal polylogarithms}

\author{Anton Alekseev\thanks{Section of Mathematics, University of Geneva, Rue du Conseil-Général 7-9, 1205 Geneva, Switzerland; \href{mailto:Anton.Alekseev@unige.ch}{anton.alekseev@unige.ch}, \href{mailto:megan.howarth@unige.ch}{megan.howarth@unige.ch}, \href{mailto:Pavol.Severa@unige.ch}{pavol.severa@unige.ch}.}, \hskip 0.3 cm 
Megan Howarth\footnotemark[1], \hskip 0.3 cm 
%\thanks{Section of Mathematics, University of Geneva, Rue du Conseil-Général 7-9, 1205 Geneva, Switzerland; \href{mailto:megan.howarth@unige.ch}{megan.howarth@unige.ch}.}, \hskip 0.3 cm 
Florian Naef\thanks{School of Mathematics, Trinity College, Dublin 2, Ireland; \href{mailto:naeff@tcd.ie}{naeff@tcd.ie}.}, \\ 
Muze Ren\thanks{Institute of Mathematics, University of Zurich, Winterthurerstrasse 190, 8057 Zürich, Switzerland; \href{mailto:muze.ren@unige.ch}{muze.ren@unige.ch}.}, \hskip 0.3 cm 
Pavol \v{S}evera\footnotemark[1]
%\thanks{Section of Mathematics, University of Geneva, Rue du Conseil-Général 7-9, 1205 Geneva, Switzerland; \href{mailto:Pavol.Severa@unige.ch}{pavol.severa@unige.ch}.}
}

\maketitle

\begin{abstract}
    Polylogarithmic functions (polylogs) in $n$ variables can be viewed as elements of $(U\mathfrak{p}_{m})^*$, the dual of the universal enveloping algebra of the Lie algebra $\mathfrak{p}_{m}$ of infinitesimal spherical pure braids with $m=n+3$ strands. Polylogs with %$n=1, 2$
    $m=4,5$ are used in the theory relating double shuffle relations and Drinfeld associators \cite{furusho_double_2011}.
We give explicit formulas for elements of $(U\mathfrak{p}_{m})^*$ representing polylogs, and compute the left ideal $J_{m} \subset U\mathfrak{p}_{m}$ given by their joint kernel. We introduce Lie subalgebras $\mathfrak{k}_{m}=\mathfrak{p}_{m} \cap J_{m}$, and we compute them for $m=4, 5$.
    
    %The joint kernel of polylogs is given by a certain left ideal $J_n \subset U\mathfrak{p}_{n+3}$.
    %We show that $J_n$ can be represented as a sum of the left ideal $U\mathfrak{p}_{n+3} \mathfrak{h}_n$ spanned by the Lie subalgebra $\mathfrak{h}_n \cong \mathfrak{p}_{n+2} \oplus \mathbb{C}$ (here $\mathbb{C}$ stands for the one dimensional Lie algebra over $\mathbb{C}$), and of an interesting two sided ideal $I_n \subset U\mathfrak{p}_{n+3}$. We define
    %the space $\mathfrak{pl}_{n+3} = \mathfrak{p}_{n+3}/(\mathfrak{p}_{n+3} \cap J_n)$, and we compute $\mathfrak{pl}_{4}$ and $\mathfrak{pl}_{5}$. The result for $\mathfrak{pl}_{5}$ finds applications in the description of the reduced coaction Lie algebra $\mathfrak{rc}_0$ and its relations to the double shuffle Lie algebra $\mathfrak{dmr}_0$.
    
\end{abstract}

\section{Introduction}

One variable polylogarithmic functions (polylogs) were considered by Bernoulli and Leibniz at the end of the 17th century, and then by Spence in the 19th century. Polylogarithms in $n$ variables were defined by Goncharov \cite{chatterji_polylogarithms_1995}. They are given by the power series
\begin{equation}       \label{eq:intro L}
L_a(x_1,\dots,x_n) = \sum_{0< k_1 < \dots <k_n} \, \frac{{x_1}^{k_1}\dots {x_n}^{k_n}}{k_1^{a_1} \dots k_n^{a_n}},
\end{equation}
where $a=(a_1, \dots, a_n)$ with $a_i \in \mathbb{Z}_{\geq 1}$.
Polylogarithmic functions satisfy certain differential equations intimately related to the Knizhnik-Zamolodchikov equations.

An algebraic viewpoint on polylogarithmic functions is given by Brown's version of the Chen iterated integral construction \cite{brown_multiple_2009}. In more detail, they are images of formal polylogs 
$$
l_{a} \in \mathcal{V}(\mathcal{M}_{0,m}),
$$
where $\mathcal{V}(\mathcal{M}_{0,m})$ is the reduced bar construction of the moduli space of rational curves with $m=n+3$ marked points. The simplest Chen's integrals are of the form
$$
l_\emptyset \mapsto 1 = L_\emptyset(x), \hskip 0.3cm
l_{(1)} \mapsto \int_0^x \frac{ds}{1-s} = -\log(1-x) =
\sum_{k=1}^\infty \frac{x^k}{k} = L_{(1)}(x).
$$
The bar construction $\mathcal{V}(\mathcal{M}_{0,m})$ can be identified with the dual of the universal enveloping algebra of the Lie algebra of infinitesimal spherical pure braids with $m$ strands,
$$
\mathcal{V}(\mathcal{M}_{0,m}) \cong (U\mathfrak{p}_m)^*.
$$

In this note, we compute the joint kernel of formal polylogs
$$
J_m:={\textstyle \bigcap_a} {\rm ker}(l_a) \subset U\mathfrak{p}_m,
$$
and show that it is a left ideal in $U\mathfrak{p}_m$. We give a description of this ideal, and an explicit formula for formal polylogs $l_a$ on a complement $C_m$ of 
$J_m$. Furthermore, we define Lie subalgebras 
$$
\mathfrak{k}_m = J_m \cap \mathfrak{p}_m,
$$
and compute them for $m=4, 5$.

The methods used in the proofs include quadratic duality between Lie algebras and exterior algebras, differential equations for polylogarithmic functions, combinatorics of monomial bases, and some basic homological algebra. This work is motivated by the use of formal polylogs in the theory of double shuffle Lie algebras \cite{furusho_double_2011}, and by their relations to the Kashiwara-Vergne theory (see \cite{schneps_double_2012, schneps_double_2025, enriquez_double_2025}). 

The structure of the paper is as follows. In Section \ref{sec:results}, we collect some background material and state the main results.
In Section \ref{sec: quadratic duality}, we recall the notion of quadratic duality for Lie algebras and some properties of modules with cyclic vector over such Lie algebras. In Section \ref{sec: sperical braids}, we construct the left ideal $J_{m} \subset U\mathfrak{p}_{m}$. In Section \ref{sec: polylogs}, we define the space of formal polylogs and show that their joint kernel is equal to
$J_{m}$. In Section \ref{sec: polylogs and free},
we define Lie algebras $\mathfrak{k}_{m}$, and compute them for $m=4, 5$.

\vskip 0.2cm

{\bf Acknowledgments.} This paper is part of the proceedings of the working group 
held by F.~Brown and the authors at the University of Geneva from November 2023 to June 2024. We are grateful to B. Enriquez, H. Furusho, and N. Markarian for very interesting discussions and for comments on the draft of this paper.

Research of A.A. was supported in part by the grants 208235, 220040, and 236683, and by the National Centre for Competence in Research (NCCR) SwissMAP of the Swiss National Science Foundation (SNSF), and by the award of the Simons Foundation to the Hamilton Mathematics Institute of the Trinity College Dublin under the program ``Targeted Grants to Institutes''. M.H. acknowledges the support of the SNSF grant $200020-200400$ and of the NCCR SwissMAP. Research of M.R. is supported by the SNSF postdoc mobility grant  $P500PT\_230340$, and he thanks Prof.~Anna Beliakova for providing the office where part of this work was carried out.

\section{Background and main results}      \label{sec:results}
Let $\mathcal{M}_{0, m}$ be the moduli space of rational curves with $m=n+3$ marked points
\begin{equation}
\label{eq:configurations-labels}
z_\alpha=0, z_1, \dots, z_n, z_\beta=1, z_\omega=\infty,
\end{equation}
indexed by the ordered set 
\begin{equation*}
\label{eq:S}
    \mathcal{S} :=\{\alpha,1,\ldots, n,\beta,\omega\}, \quad \alpha < 1 < \ldots < n < \beta <\omega.
\end{equation*}
The three extra points $z_\alpha= 0, z_\beta=1, z_\omega=\infty$ fix the symmetry under Möbius transformations. The cohomology $H^\bullet(\mathcal{M}_{0, m})$ is generated by the 1-forms
$$
\omega_{ij}= d \, \log(z_i-z_j),\hskip 0.3cm i\ne j\in \mathcal{S}\setminus\{\omega\}\quad (\omega_{\alpha\beta} = 0)
%\{1,2,\ldots,n,\alpha,\beta,\omega\}
$$
modulo quadratic Arnold relations \cite{arnold_cohomology_1969},
\begin{equation}
\label{eq:Arnold}
\omega_{ij}\wedge\omega_{ik}+\omega_{ik}\wedge \omega_{jk}+\omega_{jk}\wedge \omega_{ij}=0.
\end{equation}

Brown's version of the Chen iterated integral construction \cite{brown_multiple_2009} associates a holomorphic function
on the universal cover of $\mathcal{M}_{0,m}$ to any element 
$$
[\omega_1| \dots |\omega_r], \hskip 0.3cm \omega_i \in \{ \omega_{ij}\}
$$
of the reduced bar construction $\mathcal{V}(\mathcal{M}_{0,m})$. In particular, for any
$$a = (a_1, \dots, a_n) \in \Z_{\geq1}^n $$
there is a ``formal polylog'' $l_{a} \in \mathcal{V}(\mathcal{M}_{0,m})$ which corresponds to the polylogarithmic function 
$L_a(x_1,\dots,x_n)$ given by the power series \eqref{eq:intro L},
where $z_i = x_i \dots x_n$ ($1\leq i\leq n$), i.e.\ $x_i = z_i/z_{i+1}$, and where we set $z_{n+1}=z_\beta=1$.

More generally, if $\kappa\colon\{1,\dots, N\} \to \{1,\dots, n\}$ is a non-decreasing function and if $a\in\Z_{\geq1}^N$, we can consider the polylogarithmic function
$$L_{a,\kappa}(x_1,\dots,x_n) = L_a(y_1,\dots, y_N)$$
where
$$y_i = z_{\kappa(i)} / z_{\kappa(i+1)} = \prod_{\kappa(i)\leq j <\kappa(i+1)} x_j \qquad(\text{with }\kappa(N+1) := \beta).$$
Again, 
there is a unique element $l_{a,\kappa} \in \mathcal{V}(\mathcal{M}_{0,m})$ corresponding to this function.

The bar construction $\mathcal{V}(\mathcal{M}_{0,m}) \cong (U\mathfrak{p}_{m})^*$ is the  (graded) dual of the universal enveloping algebra of the Lie algebra of infinitesimal spherical pure braids $\mathfrak{p}_{m}$, defined by generators $X_{ij}=X_{ji}, i \neq j$, and relations
$$
\sum_{j,\, j\neq i} X_{ij}=0, \hskip 0.3cm [X_{ij}, X_{kl}]=0,
$$
where $i,j,k,l \in \mathcal{S}$ are all distinct. Formal polylogs $l_{a,\kappa}$ span an interesting subspace
$$
\mathcal{L}_{m} \subset (U\mathfrak{p}_{m})^*.
$$

In this paper we give an explicit description of formal polylogs $l_{a,\kappa}$ as elements of $(U\mathfrak{p}_{m})^*$ and of their joint kernel $(\mathcal{L}_{m})^\perp$, which is a left ideal in $U\mathfrak{p}_{m}$. 
%We label generators of $\mathfrak{p}_{n+3}$ with indices $\alpha, 1, \dots, n, \beta, \omega$. 
To do it, we need the short exact sequence of Lie algebras
$$
0 \to \mathfrak{f}_{n+1} \to \mathfrak{p}_{m} \overset{\pi_{\beta}}{\to} \mathfrak{p}_{m-1} \to 0,
$$
where the map $\pi_\beta\colon \mathfrak{p}_{m} \to \mathfrak{p}_{m-1}$ corresponds to forgetting the label $\beta$, and the free Lie algebra $\mathfrak{f}_{m-2}$ has generators
$X_{i \beta}, i=1, \dots, n$, and $X_{\beta \omega}$. We will also need a splitting $s\colon \mathfrak{p}_{m-1} \to \mathfrak{p}_{m}$ of this exact sequence (for more details, see Section \ref{subsection:section-map}).

The main results of this note are as follows:

\begin{Th} \label{th: intro}
    The joint kernel of formal polylogs $l_{a,\kappa} \in \mathcal{L}_{m}$ is the left ideal
    \begin{equation*}
        \mathcal{L}_{m}^\perp =
        U\mathfrak{p}_{m}\, s(\mathfrak{p}_{m-1}) +  U\mathfrak{p}_{m}\,X_{\beta \omega} 
        + \sum_{1\leq i<j \leq n} (X_{i\beta})(X_{j\beta}),
    \end{equation*}
    where $(X_{i\beta})\subset U\mathfrak{p}_{m}$ is the two sided ideal generated by $X_{i\beta}$. 
\end{Th}

\begin{Th} \label{th: intro2}
    Under the direct sum decomposition 
    $$U\mathfrak{p}_{m} = U\mathfrak{f}_{m-2} \oplus U\mathfrak{p}_{m}\, s(\mathfrak{p}_{m-1})$$
     the formal polylogs $l_{a,\kappa}\in (U\mathfrak{p}_{m})^*$ vanish on $U\mathfrak{p}_{m}\, s(\mathfrak{p}_{m-1})$ and for any
     $$\alpha \in U\mathfrak{f}_{m-2},$$ $l_{a,\kappa}(\alpha)$ is $(-1)^N$ times the coefficient of the monomial 
     $$w_{a,\kappa}: =\prod_{i=N}^1 X_{\beta\omega}^{a_i - 1} X_{\beta\kappa(i)} \in U\mathfrak{f}_{m-2}$$
      in $\alpha$. 
\end{Th}

Furthermore, we define the Lie subalgebras
$$
\mathfrak{k}_{m}= \mathcal{L}_{m}^\perp \cap \mathfrak{p}_{m} \subset \mathfrak{p}_{m},
$$
and compute them for $m=4, 5$.

\begin{Th}       \label{th: intro3}
  We have, 
  $$
\mathfrak{k}_4 = \mathbb{C} X_{\beta \omega}, \hskip 0.3cm
\mathfrak{k}_5 = \mathbb{C} X_{\beta \omega} \oplus [\lambda, \lambda] \oplus s(\mathfrak{p}_4),
  $$
  where $\lambda \subset \mathfrak{f}_3$ is the Lie ideal spanned by Lie words containing $X_{1\beta}$ and $X_{2\beta}$.
\end{Th}

\section{Quadratic duality for Lie algebras}      \label{sec: quadratic duality}

In this Section, we recall the notion of quadratic duality for Lie algebras and outline some of its applications.

\subsection{Lie algebras defined by quadratic relations}
\label{sec:quadratic-lie}

Let $V$ be a finite-dimensional vector space, $R \subset \bw^2 V$ be a subspace in the second exterior power of $V$, and $R^\perp \subset \bw^2 V^*$ be the the annihilator of $R$:
%
% $$
% R^\perp = \{ \alpha \in \bw^2 V^*; \, \langle \alpha, a\rangle =0 \,\, \forall 
% a \in \bw^2 V\}.
% $$
$$
R^\perp = \{ \alpha \in \bw^2 V^*; \, \langle \alpha, a\rangle =0 \,\, \forall 
a \in R\}.
$$
One can associate the following constructions to the pair $(R, R^\perp)$. First, one considers the graded commutative algebra $\mathcal{A}_R$ equipped with the canonical projection $\pi_R$
\begin{equation*}
    \mathcal{A}_R = \bw V/(R), \hskip 0.3cm \pi_R: \bw V \to \mathcal{A}_R,
\end{equation*}
where $(R)$ is the ideal in $\bw V$ generated by $R$. A morphism of quadratic algebras $F:\mathcal{A}_{R_1}\to \mathcal{A}_{R_2}$, where $R_1\subset\bw^2 V_1$ and $R_2\subset\bw^2 V_2$, is a $\mathbb{C}$-algebra morphism that preserves grading. It is easy to see that such a map is completely determined by its degree one component, which is a linear map
\begin{equation}
    \label{eq:deg-1-map}
f:V_1 \to V_2, \text{ such that}~ (f\wedge f)(R_1)\subset R_2.
\end{equation}

The graded commutative algebra $\mathcal{A}_R$ has a quadratic dual graded Lie algebra $\mathbb{L}_{R^{\perp}}$,
\begin{equation*}
\mathbb{L}_{R^{\perp}} = \mathbb{L}(V^*)/(R^\perp)_{\rm Lie}, \hskip 0.3cm 
(R^\perp)_{\rm Lie} = R^\perp + [R^\perp, \mathbb{L}(V^*)],
\end{equation*}
where $\mathbb{L}(V^*)$ is the free Lie algebra spanned by the vector space $V^*$, and $(R^\perp)_{\rm Lie}$ is the Lie ideal in $\mathbb{L}(V^*)$ generated by $R^\perp$. The grading is defined by assigning degree $1$ to elements of $V^* \subset \mathbb{L}_R$.
The universal enveloping algebra
\begin{equation*}
    U \mathbb{L}_R \cong {\rm Ass}(V^*)/(R^\perp)
\end{equation*}
is isomorphic to the free associative algebra ${\rm Ass}(V^*)$ spanned by $V^*$ modulo the two-sided ideal $(R^\perp)$ generated by $R^\perp$. It is a Hopf algebra with the standard cocommutative coproduct. Its (graded) dual is the commutative Hopf algebra $(U \mathbb{L}_{R})^*$, which can be described as a subspace of %the free associative algebra 
${\rm Ass}(V)$ equipped with the shuffle product and the deconcatenation coproduct. 
This subspace is
given by the condition that the projection maps 
$$
\wedge_k: V^{\otimes n}\xrightarrow{\id_{V^{\otimes k}} \otimes \, (\pi_R \circ \wedge^2) \otimes\id_{V^{\otimes l}}}
V^{\otimes k}\otimes \mathcal{A}_R \otimes V^{\otimes l}
%\qquad(k+l+2=n)
$$
vanish on elements of $(U \mathbb{L}_{R^\perp})^*$ for all $k+l=n-2$.
As a vector space, ${\rm Ass}(V)$ is spanned by monomials
$[\omega_1 | \omega_2 | \dots |\omega_l]$ with $\omega_i \in V$.
Elements of $(U \mathbb{L}_{R})^*$ are linear combinations of these
monomials contained in the joint kernel of $\wedge_k$'s.\\

%\subsection{Differentials and canonical elements}

For what follows, it is convenient to introduce a canonical element in the algebra $U\mathbb{L}_R \otimes \mathcal{A}_R$:
$$
\Omega = \sum_a e^a \otimes e_a \in V^* \otimes V \subset  U\mathbb{L}_R \otimes \mathcal{A}_R,
$$
where $\{ e_a\}$ is a basis of $V$ and $\{ e^a\}$ is the dual basis of $V^*$. Note that $[\Omega, \Omega]=0$. 

\begin{Prop} \label{prop:omegas}
Let $\Omega_i$ be a canonical element of $U\mathbb{L}_{R_i} \otimes \mathcal{A}_{R_i}$, for $i=1,2$. Let $f:V_1 \to V_2$ denote the degree one component linear map as in \eqref{eq:deg-1-map}, and $f^{*}$ its dual map. It holds
    \begin{equation} \label{eq:omegas}
        (f^{*}\otimes 1)\Omega_2 = (1 \otimes f)\Omega_1.
    \end{equation}
\end{Prop}

\begin{proof}
It is easy to see that both sides of \eqref{eq:omegas} are equal to $f \in V_1^{*}\otimes V_2$.
\end{proof}

We also introduce a differential on the space 
$(U\mathbb{L}_R)^* \otimes \mathcal{A}_R$,
%
% \begin{equation}\label{eq: d}
% d \in {\rm End}((U\mathbb{L}_R)^* \otimes \mathcal{A}_R);\hskip 0.3cm
% [\omega_1| \dots|\omega_k] \otimes \alpha \mapsto
% [\omega_1| \dots|\omega_{k-1}] \otimes (\omega_k \wedge \alpha)
% \end{equation}
\begin{equation}\label{eq: d}
\begin{split}
d : (U\mathbb{L}_R)^* \otimes \mathcal{A}_R & \to (U\mathbb{L}_R)^* \otimes \mathcal{A}_R \\
[\omega_1| \dots|\omega_k] \otimes \alpha & \mapsto
[\omega_1| \dots|\omega_{k-1}] \otimes (\omega_k \wedge \alpha).
\end{split}
\end{equation}

\begin{Prop}
 The differential \eqref{eq: d} satisfies
 $$
%[\Omega, \Omega]=0, \hskip 0.3cm 
d^2=0.
 $$
\end{Prop}

\begin{proof}
%     For the first equality, we compute
%     %
%     $$
% [\Omega, \Omega] = \sum_{a,b} [e^a, e^b] \otimes e_a \wedge e_b \in  (\bw^2 V^*/R^\perp) \otimes (\bw^2 V/R).
%     $$
%     The sum over $a,b$ is the canonical element in $ \bw^2 V^* \otimes \bw^2 V $. Since we quotient these spaces by the subspace $R$ and its annihilator $R^\perp$, the result vanishes, as required.
% For the second equality, we compute
We compute
    $$
d^2: [\omega_1| \dots|\omega_k] \otimes \alpha \mapsto
[\omega_1| \dots|\omega_{k-2}] \otimes (\omega_{k-1} \wedge \omega_k \wedge \alpha) =0,
    $$
    %Here, we have used that 
    since $\omega_{k-1} \wedge \omega_k = 0$ 
    in $\mathcal{A}_R$.

\end{proof}

\subsection{Modules and actions}

In this Section, we discuss $\mathbb{L}_R$-modules and the ways to define them.
As in the previous Section, let $V$ be a vector space, $R \subset \bw^2 V$, $R^\perp \subset \bw^2 V^*$, $\mathbb{L}_{R}$ the Lie algebra defined in the previous section, and $M$ an $\mathbb{L}_{R}$-module. Since the Lie algebra $\mathbb{L}_{R}$ is generated by elements of $V^*$, the $\mathbb{L}_{R}$-action on $M$ is completely defined by either the map
$$
a: V^* \otimes M \to M,
$$
or alternatively by its transpose $\delta$,
$$
\delta: M^* \to M^* \otimes V.
$$
% To the map $a$, one can associate a map
% %
% $$
% a \circ (1 \otimes a): (V^*)^{\otimes 2} \otimes M \to M,
% $$
% and 
To the map $\delta$ one can associate a map 
$d_M \in {\rm End}(M^* \otimes \mathcal{A}_R)$ defined by the formula
$$
d_M: m^* \otimes \alpha  \mapsto (\delta(m^*))' \otimes (\delta(m^*))'' \wedge \alpha,
$$
where we have used the Sweedler notation for $\delta(m^*) \in M^* \otimes V $.

\begin{Ex}
    $U\mathbb{L}_{R}$ is an $\mathbb{L}_{R}$-module under multiplication on the left with the map $a: x \otimes u \mapsto xu$. The dual map $\delta$ is given by the deconcatenation
    $$
[\omega_1| \dots| \omega_k] \mapsto [\omega_1| \dots| \omega_{k-1}]
\otimes \omega_k,
    $$
    and the corresponding map $d_{U\mathbb{L}_{R}}$ is given by equation \eqref{eq: d}.
\end{Ex}

\begin{Prop}       \label{prop: key}
For a map $a: V^* \otimes M \to M$ and its dual $\delta$, the following statements are equivalent:
\begin{enumerate}[label=(\roman*)]
    \item \label{prop:map-a} The map $a$ defines an $\mathbb{L}_{R}$-action on $M$.
    \item \label{prop:map-delta} The map $d_M$ associated to $\delta$ is a differential, $d_M^2=0$.
\end{enumerate}
 
\end{Prop}

\begin{proof}
To show the equivalence of \ref{prop:map-a} and \ref{prop:map-delta}, observe that the map $a$ defines an action on $M$ of the free Lie algebra generated by $V^*$. This action descends to an action of $\mathbb{L}_R$ if and only if the relations $R^\perp \subset \bw^2 V^*$ act by zero on $M$, and this is equivalent to $R^\perp \otimes M \subset {\rm ker}(a \circ (1_{V^{*}} \otimes a))$, where
$$
a \circ (1_{V^{*}} \otimes a): (V^*)^{\otimes 2} \otimes M \to M.
$$

The map $d_M$ squares to zero if and only if the map
$$
(1_{M^{*}} \otimes \wedge)\circ ( \delta \otimes 1_V) \circ \delta: M^* \to M^* \otimes  \bw^2 V
$$
takes values in $M^* \otimes R$, and this statement is the dual of $R^\perp \otimes M \subset \ker \left(a \circ (1_{V^{*}} \otimes a) \right)$.
   
\end{proof}

%\begin{Prop}     \label{prop:modules functorial}
%    For two $U\mathbb{L}_R$-modules $M_1, M_2$, a linear map $\phi: M_1 \to M_2$ is a module morphism if and only if the induced map $M_2^* \otimes \mathcal{A}_R \to M_1^* \otimes \mathcal{A}_R$ is a chain map with respect to the differentials $d_{M_1}$ and $d_{M_2}$.
%\end{Prop}

%\begin{proof}
   
%\end{proof}

For future use, we also recall the definition and basic properties of modules with a cyclic vector.
Let $\mathfrak{g}$ be a Lie algebra and $M$ be a $\mathfrak{g}$-module (or, equivalently, a $U\mathfrak{g}$-module). We say that $v \in M$ is a cyclic vector if $U\mathfrak{g} \cdot v = M$. There is an equivalence between $\mathfrak{g}$-modules with a cyclic vector and left ideals in $U\mathfrak{g}$; in more detail, to a left ideal $I \subset U\mathfrak{g}$ corresponds the $\mathfrak{g}$-module $M_I=U\mathfrak{g}/I$ with the cyclic vector $v = I \in M_I$, and to a $\mathfrak{g}$-module with a cyclic vector $(M, v)$ corresponds the left ideal
$$
I_{(M, v)} = \{ a \in U\mathfrak{g}; \, a \cdot v =0\}.
$$
From now on, as will be the case in subsequent parts of this note, we assume the Lie algebra $\mathfrak{g}$ and the $\mathfrak{g}$-module $M$ to be graded with finite-dimensional graded components. For a module with a cyclic vector $(M, v)$, the (graded) dual module $M^*$ can be realized as a subspace of $(U\mathfrak{g})^*$:
$$
M^* \cong \{ \alpha \in (U\mathfrak{g})^*; \, \langle \alpha, a\rangle =0 \,\, \forall a \in I_{(M,v)}\}.
$$

\begin{Prop}      \label{prop:new}
    Let $\mathcal{L} \subset (U\mathfrak{g})^*$ be a $U\mathfrak{g}$-submodule. Then, the dual module $\mathcal{L}^*$ has a cyclic vector, and the corresponding left ideal is equal to $\mathcal{L}^\perp$.
\end{Prop}

\begin{proof}
Let $\pi: U\mathfrak{g} \to \mathcal{L}^*$ be the projection dual to the injection
$\mathcal{L} \rightarrowtail (U\mathfrak{g})^*$. Since the latter is a $U\mathfrak{g}$-module morphism, so is the former. Hence, $v = \pi(1)$ is a cyclic vector in $\mathcal{L}^*$. The corresponding left ideal is the kernel of $\pi$, which coincides with $\mathcal{L}^\perp$, as required.
\end{proof}

\section{Lie algebras $\mathfrak{p}_m$ and $U\mathfrak{p}_{m}$-modules}
\label{sec: sperical braids}

In this Section, we define the Lie algebras of infinitesimal spherical pure braids $\mathfrak{p}_m$  and construct several $U\mathfrak{p}_{m}$-modules
using the maps $a$ and $\delta$ (as in Proposition \ref{prop: key}). 

\subsection{Lie algebras $\mathfrak{p}_m$}

As an example of the construction of the previous section, we consider the moduli space of genus zero curves with $m$ marked points
$\mathcal{M}_{0, m}$. It is a standard fact that the cohomology $ \mathcal{A}=H^\bullet(\mathcal{M}_{0, m})$ is generated by $V = H^1(\mathcal{M}_{0, m})$ modulo quadratic Arnold relations $R$ (see \eqref{eq:Arnold}). It is convenient to label configurations using the set $\mathcal{S}=\{ \alpha, 1, \dots, n, \beta, \omega\}$; that is,  
    $$
z_\alpha=0, z_1, \dots, z_n, z_\beta=1, z_\omega=\infty,
    $$
    where $n=m-3$. The dimension of $V$ is equal to $n(n+3)/2$, and one possible choice of generators is as follows:
\begin{align*}
\omega_{\alpha i} & = \frac{dz_i}{z_i};  \quad %\hskip 0.3cm 
\omega_{i \beta} = \frac{dz_i}{z_i -1}, \hskip 0.3cm 1\le i\le n; \\ 
\omega_{ij} & = \frac{ d(z_i - z_j)}{z_i - z_j}, \hskip 0.3cm 1\le  i<j\le n.
\end{align*}
It is sometimes more convenient to use coordinates $(x_1, \dots, x_n)$ and their products $u_{ij}$ defined by
\begin{equation}      \label{eq: xyz}
z_i = x_i \dots x_n, \hskip 0.3cm u_{ij} = x_i \dots x_{j-1}.
\end{equation} 
In this notation, we have
$$
\omega_{\alpha i} = \sum_{k=i}^n \frac{dx_k}{x_k}, \hskip 0.3cm
\omega_{ij} = \sum_{k=j}^n \frac{dx_k}{x_k} +\frac{du_{ij}}{u_{ij}-1}.
$$

    The dual Lie algebra $\mathbb{L}_{R^{\perp}}$ is the Lie algebra of infinitesimal spherical pure braids $\mathfrak{p}_{m}$ with generators $X_{ab}=X_{ba}, a\neq b$, and relations
    \begin{equation}     \label{eq:relations p}
[X_{ab}, X_{cd}]=0, \hskip 0.3cm \sum_{b \neq a} X_{ab} =0,
    \end{equation}
    where $a,b,c,d \in \mathcal{S}$ are all distinct.
    One can choose $n(n+3)/2$ linearly independent generators of $V^*$:
$$
%\begin{align*}
 % & 
  X_{\alpha i},X_{i \beta}, \hskip 0.3cm 1\le i\le n; \\ %i=1, \dots, n; \
  %& 
  \hskip 0.3cm
  X_{ij}, \hskip 0.3cm 1 \leq i <j \leq n.  
%\end{align*}
$$
It is easy to see that relations \eqref{eq:relations p} imply
\begin{equation*} \label{eq:4T}
[X_{ac}+X_{bc}, X_{ab}]=0.
\end{equation*}
Sometimes, it will also be convenient to use the generator
$$
X_{\beta \omega} = \sum_i X_{\alpha i} + \sum_{i<j} X_{ij}
$$
instead of $X_{\alpha n}$. 
Observe that 
%$X_{\alpha i},X_{ij}$ are centralizers of $X_{\beta \omega}$, 
the Lie brackets of $X_{\beta \omega}$ with  generators $X_{\alpha i}, X_{ij}$ vanish. 
The dual Hopf algebra $(U \mathfrak{p}_{m})^*$ is the reduced bar algebra
    $\mathcal{V}(\mathcal{M}_{0, m})$. Its elements are presented as linear combinations of expressions
    $
[\omega_1| \dots| \omega_k],
    $
    where $\omega_i$'s are generators of $H^1(\mathcal{M}_{0, m})$.

\subsection{Lie subalgebras of $\mathfrak{p}_{m}$}
\label{subsection:section-map}

For each choice of %$t \in \{ 1, \dots, m\}$ 
$t \in \mathcal{S}$, there is a short exact sequence of Lie algebras
$$
0 \to \mathfrak{f}_{m-2} \to \mathfrak{p}_m \to \mathfrak{p}_{m-1} \to 0.
$$
Here, the free Lie algebra $\mathfrak{f}_{m-2}$ is generated by
$X_{tv}$, $v\neq t$, subject to the linear relation $\sum_{v \neq t} X_{tv}=0$; the map $ \pi_t : \mathfrak{p}_m \to \mathfrak{p}_{m-1}$ maps generators $X_{tv}, v\neq t$ to zero. 
% This map admits a section $s_w$ for each 
% $w \neq t$. In more detail, $s_w$ is an injective Lie homomorphism given by
% %
% $$
% s_w: 
% \begin{cases}
%     x_{ab} \mapsto x_{ab} & {\rm if} \,\, a,b \neq w, \\
%     x_{aw} \mapsto x_{aw} + x_{au} &
% \end{cases}
% $$
% It can also be viewed as a ``cabling map'' replacing the ``strand'' $w$ with two strands $t$ and $w$.
In what follows,  we choose $t=\beta$ and fix the generators of $\mathfrak{f}_{m-2}$ to be  $X_{i\beta}, i=1, \dots, n$ and $X_{\beta \omega}$. 
The map $ \pi_\beta : \mathfrak{p}_m \to \mathfrak{p}_{m-1}$ admits a section $s : \mathfrak{p}_{m-1} \to \mathfrak{p}_{m}$, which is defined as follows: labeling  $m-1$ indices of $\mathfrak{p}_{m-1}$ by $ \{\alpha,1,\ldots,n,\omega \}$, $s$ sends the generators $X_{\alpha i}$ $(1 \leq i \leq n-1)$ and $X_{ij}$ $(1 \leq i,j \leq n)$ of $\mathfrak{p}_{m-1}$ to the corresponding generators of $\mathfrak{p}_{m}$. 

\begin{Rem}
Geometrically, this section map corresponds to the homotopy class of maps
\begin{align*}
    \mathcal{M}_{0,m-1} & \to \mathcal{M}_{0,m} \\
    (z_{\alpha} = 0, z_1, \ldots, z_{n-1}, z_n =1, z_\omega = \infty) & \mapsto (z_{\alpha}, z_1, \ldots, z_{n}, z_\beta = 1 + \varepsilon(z_1,\dots, z_n)  , z_\omega)
\end{align*}
where $\varepsilon$ is an arbitrary real continuous function of the $z_i$'s such that for all $1\leq i\leq n-1$, $0<\varepsilon<|1-z_i|$.

\end{Rem}

A derivation $D$ of a free Lie algebra $\mathfrak{f}(x_1, \dots, x_k)$ is called tangential if $D(x_i)=[x_i, v_i]$ for some elements $v_i \in \mathfrak{f}(x_1, \dots, x_k)$. %$D$ is an inner derivation if $v_i=v$ for all $i$. 

\begin{Prop}       \label{prop: tangential action}
   The action of $s(\mathfrak{p}_{m-1})$ on $\mathfrak{f}(X_{1\beta}, \dots, X_{n\beta}, X_{\beta\omega})$ is by tangential derivations. The generator $X_{\beta \omega}$ is in the kernel of this action. 
\end{Prop}

\begin{proof}
Recall that the section $s$ is determined by $s(X_{\alpha i})=X_{\alpha i}$, $s(X_{ij})=X_{ij}$, for $1\le i\ne j\le n$. 
The Lie subalgebra $s(\mathfrak{p}_{m-1}) \subset \mathfrak{p}_m$ acts on $\mathfrak{f}(X_{1\beta},\ldots,X_{n\beta}, X_{\beta\omega})$ by the adjoint action. We have the relations
$[X_{\alpha i},X_{\beta\omega}]=0$ and $ [X_{ij},X_{\beta\omega}]=0$, therefore $X_{\beta\omega}$ is in the kernel of this action.

For the other generators $X_{k\beta}$ with  $1\le k\le n$ the action is
\begin{align*}
X_{\alpha i}: X_{k\beta}\mapsto &\delta_{ik}[X_{\alpha k},X_{k\beta}]=-\delta_{ik}[X_{\alpha\beta},X_{k\beta}],\\
X_{ij}: X_{k\beta}\mapsto &\delta_{ik}[X_{kj},X_{k\beta}]+\delta_{jk}[X_{ik},X_{k\beta}]=-[\delta_{ik} X_{j\beta}+\delta_{jk}X_{i\beta},X_{k\beta}],
\end{align*}
which implies that the action is by tangential derivations.

\end{proof}

For future use, we consider in more detail the duality between Arnold 1-forms and generating sets of the Lie algebras $\mathfrak{p}_m, \mathfrak{f}_{m-2}$ and $s(\mathfrak{p}_{m-1})$. For $\mathfrak{p}_m$, the canonical element $\Omega$ takes the form
\begin{align*}
\Omega ={}&
\sum_{i=1}^n\left( X_{\alpha i} \frac{dz_i}{z_i} 
+ X_{i \beta} \frac{dz_i}{z_i -1 }\right) + \sum_{i<j} X_{ij} \frac{d(z_i-z_j)}{z_i-z_j} \\
 ={}&
\sum_{i=1}^n\left( X_{\alpha i} \sum_{k=i}^n \frac{dx_k}{x_k}
+ X_{i \beta} \frac{dz_i}{z_i -1 }\right) + \sum_{i<j} X_{ij} \left(\sum_{k=j}^n \frac{dx_k}{x_k} + \frac{du_{ij}}{u_{ij}-1}\right) \\
 ={}&
\sum_{i=1}^{n-1}  X_{\alpha i} \left( \sum_{k=i}^{n-1} \frac{dx_k}{x_k} \right)
 + \sum_{i<j} X_{ij} \left(\sum_{k=j}^{n-1} \frac{dx_k}{x_k} + \frac{du_{ij}}{u_{ij}-1}\right) \\
& +  \sum_{i=1}^n X_{i \beta} \frac{dz_i}{z_i -1 } + X_{\beta \omega} \frac{dx_n}{x_n}.  
\end{align*}
Here, we have used the notation \eqref{eq: xyz}: $u_{ij}=x_i \dots x_{j-1}, z_i =x_i \dots x_n$. 
As an application of Proposition \ref{prop:omegas}, and more generally Section \ref{sec:quadratic-lie}, consider the injective Lie homomorphism $\mathfrak{f}_{m-2} \to \mathfrak{p}_m$. In this context, 
\begin{equation} \label{eq:V1-pm}
\begin{split}
    V^{*}_1 & = \mathbb{C}\langle X_{\alpha i}, X_{\beta i}, i=1, \ldots, n; X_{ij}, 1 \leq i < j \leq n\rangle, \\
    V_1 & = \mathbb{C}\langle \omega_{\alpha i}, \omega_{\beta i}, i=1, \ldots, n; \omega_{ij}, 1 \leq i < j \leq n\rangle
\end{split}
\end{equation}
and
\begin{equation*}
\begin{split}
    V^{*}_2 & = \mathbb{C}\langle X_{\beta i}, i=1, \ldots, n; X_{\beta \omega}\rangle, \\
    V_2 & = \mathbb{C}\langle \omega_{\beta i}, i=1, \ldots, n; \omega_{\beta \omega}\rangle.
\end{split}
\end{equation*}
We now compute the LHS and RHS of \eqref{eq:omegas}: the LHS yields
\begin{equation*}
    (f^{*} \otimes 1) \left(\sum_{i=1}^{n} X_{i \beta} \omega_{i \beta} + X_{\beta \omega} \omega_{\beta \omega}\right) = \sum_{i=1}^{n} X_{i \beta} \omega_{i \beta} + X_{\beta \omega} \omega_{\beta \omega}
\end{equation*}
and the RHS,
\begin{equation*}
\begin{split}
    (1 \otimes f)(\Omega) = \sum_{i=1}^{n-1}  X_{\alpha i} \left( \sum_{k=i}^{n-1} f \left(\frac{dx_k}{x_k} \right) \right)
& + \sum_{i<j} X_{ij} \left(\sum_{k=j}^{n-1} f \left(\frac{dx_k}{x_k} \right) + f \left(\frac{du_{ij}}{u_{ij}-1} \right) \right) \\
& +  \sum_{i=1}^n X_{i \beta} f \left(\frac{dz_i}{z_i -1} \right) + X_{\beta \omega} f\left(\frac{dx_n}{x_n}\right). 
 \end{split}
\end{equation*}
Comparing them, we thus deduce the following projection map $p_1$ on Arnold 1-forms,
\begin{equation*}      \label{eq: p1}
p_1: \begin{cases}
  \frac{dx_i}{x_i} \mapsto 0  &  {\rm for} \hskip 0.3cm i=1, \dots, n-1, \\
  \frac{du_{ij}}{u_{ij}-1} \mapsto 0 & {\rm for} \hskip 0.3cm 1 \leq i <j \leq n.
\end{cases}
\end{equation*}
Similarly, consider the injective Lie homomorphism $s: \mathfrak{p}_{m-1} \to \mathfrak{p}_m$. Now, $V_1$ and $V^{*}_1$ are as in \eqref{eq:V1-pm} and
\begin{equation*}
\begin{split}
    V^{*}_2 & = \mathbb{C}\langle X_{\alpha i}, i=1, \ldots, n-1; X_{ij}, 1 \leq i < j \leq n\rangle, \\
    V_2 & = \mathbb{C}\langle \omega_{\alpha i}, i=1, \ldots, n-1; \omega_{ij}, 1 \leq i < j \leq n\rangle.
\end{split}
\end{equation*}
Computing the LHS and RHS of \eqref{eq:omegas}, we get respectively,
\begin{equation*}
    (f^{*} \otimes 1) \left( \sum_{i=1}^{n-1} X_{\alpha i} \omega_{\alpha i} + \sum_{i<j} X_{ij} w_{ij} \right)= \sum_{i=1}^{n-1} X_{\alpha i} \omega_{\alpha i} + \sum_{i<j} X_{ij} w_{ij}
\end{equation*}
and 
\begin{equation*}
\begin{split}
    (1 \otimes f)(\Omega) = \sum_{i=1}^n X_{\alpha i} f \left(\frac{dz_i}{z_i} \right) 
+ \sum_{i=1}^n X_{i \beta} f \left(\frac{dz_i}{z_i -1 } \right) + \sum_{i<j} X_{ij} f \left(\frac{d(z_i-z_j)}{z_i-z_j} \right).
\end{split}
\end{equation*}
Comparing these two expressions gives rise to the projection map $p_2$
\begin{equation*}      \label{eq:p2}
p_2: 
\begin{cases}
    \frac{dz_i}{z_i -1 } \mapsto 0 & {\rm for} \hskip 0.3cm i=1, \dots, n, \\
    \frac{dx_n}{x_n} \mapsto 0. &
\end{cases}
\end{equation*}

\subsection{Left ideals in $U\mathfrak{p}_m$}

In this Section, we define a certain left ideal in $U\mathfrak{p}_m$. As a preparation, we consider the left ideal $I_m \subset U\mathfrak{f}_{m-2}$:
$$
I_m = (U\mathfrak{f}_{m-2})\, X_{\beta\omega} + \sum_{\substack{i<j \\ i,j\in \mathcal{S}\setminus \{\beta\}}} (X_{i\beta}) (X_{j\beta}),
%\hskip 0.3cm
%I_\leq = (U\mathfrak{f}_{n+1})\, X_{\beta\omega} + \sum_{i\leq j} %(X_{i\beta}) (X_{j\beta}).
$$
%where 
%$\mathfrak{f}_{n+1}=\mathfrak{f}(x_{1\beta}, \dots, x_{n\beta}, %x_{\beta\omega})$ is the free Lie algebra with $n+1$ generators. 
We also consider the subspace 
$$
C_m= \mathbb{C}\langle w_{a,\kappa} \mid N\geq 1,\,  a \in \mathbb{Z}_{\geq 1}^N,\; \kappa\text{ non-decreasing}\rangle \subset U\mathfrak{f}_{m-2},
$$
where
\begin{multline*}
w_{a,\kappa} = X_{\beta \omega}^{a_N-1} X_{\kappa(N) \beta} \dots X_{\beta \omega}^{a_1-1} X_{\kappa(1) \beta}\\
\text{with } a=(a_1, \dots, a_N) \in \mathbb{Z}_{\geq 1}^N,\ \kappa\colon\{1,\dots, N\}\to\{1,\dots, n\}.
\end{multline*}

\begin{Prop}
\label{prop:C-complement-I}
   The subspace $C_m$ is a complement of $I_m$ in $U\mathfrak{f}_{m-2}$. That is,
    \begin{equation}      \label{eq:I}
U\mathfrak{f}_{m-2} =  C_m \oplus I_m.
%U\mathfrak{f}_{n+1} =  C_\leq \oplus I_<.
\end{equation}
\end{Prop}

\begin{proof}
    $U\mathfrak{f}_{m-2}$ is a free associative algebra with generators $X_{\beta \omega}, X_{i\beta}$ for $i=1, \dots, n$. The left ideal $I_m$ is spanned by the words which end on $X_{\beta \omega}$, or which have a letter $X_{i\beta}$ to the left of the letter $X_{j\beta}$ with $i<j$. The space $C_m$ is spanned by all the words which do not meet any of these criteria. Since there are no relations between words in a free associative algebra, $C_m$ is a complement of $I_m$, as required. 
\end{proof}

Next, we define the subspace 
$$
J_m = (U\mathfrak{p}_m) s(\mathfrak{p}_{m-1}) \oplus I_m \subset U\mathfrak{p}_m.
$$

\begin{Prop}      \label{prop: complements}
    $J_m$ is a left ideal, and $C_m$ is its complement  in $U\mathfrak{p}_m$,
    $$
U\mathfrak{p}_m = C_m \oplus J_m.
    $$
\end{Prop}

\begin{proof}
    By the PBW theorem, we have
    $$
U\mathfrak{p}_m = (U\mathfrak{f}_{m-2}) (Us(\mathfrak{p}_{m-1})) =
 U\mathfrak{f}_{m-2} \oplus U\mathfrak{p}_{m} s(\mathfrak{p}_{m-1}).
    $$
    Using equation \eqref{eq:I}, we obtain
$$
U\mathfrak{p}_m  =  (C_m \oplus I_m) \oplus U\mathfrak{p}_{m} s(\mathfrak{p}_{m-1}) = C_m \oplus J_m,     
$$
as required. 

To show that the subspace $J_m$ is a left ideal in $U\mathfrak{p}_m$, we define the following left ideal:
$$
J'_m  =  (U\mathfrak{p}_m) s(\mathfrak{p}_{m-1}) + U\mathfrak{p}_m X_{\beta \omega}
+ \sum_{i < j} (X_{i\beta})_{U\mathfrak{p}_m} (X_{j\beta})_{U\mathfrak{p}_m}.
$$
Obviously, $J_m \subset J'_m$  since we have replaced left and two-sided ideals in $U\mathfrak{f}_{m-2}$ by the corresponding left and two-sided ideals in  $U\mathfrak{p}_m$. To prove inclusion in the opposite direction, observe that
$$
U\mathfrak{p}_m X_{\beta \omega} =(U\mathfrak{f}_{m-2}) (Us(\mathfrak{p}_{m-1})) X_{\beta \omega} \subset U\mathfrak{f}_{m-2} X_{\beta \omega} + (U\mathfrak{p}_m) s(\mathfrak{p}_{m-1}). 
$$
Here we have used the fact that $X_{\beta\omega}$ is annihilated by the adjoint action of $s(\mathfrak{p}_{m-1})$. Next, observe that by Proposition \ref{prop: tangential action}, the action of $s(\mathfrak{p}_{m-1})$ on $\mathfrak{f}_{m-2}$ is tangential. Hence, we have
$$
(X_{i\beta})_{U\mathfrak{p}_m} \subset (X_{i\beta})_{U\mathfrak{f}_{m-2}}(Us(\mathfrak{p}_{m-1})) \subset (X_{i\beta})_{U\mathfrak{f}_{m-2}}  + 
(U\mathfrak{p}_m) s(\mathfrak{p}_{m-1}).
$$
The two arguments above show that $J'_m \subset J_m$, and this concludes the proof.
\end{proof}

\section{Polylogarithms}      \label{sec: polylogs}

In this Section, we define formal polylogarithms as elements of $(U\mathfrak{p}_m)^*$, and show that their joint kernel is exactly the left ideal $J_m$.

\subsection{Definition of polylogarithms}
Recall that configurations in $\mathcal{M}_{0,m}$ with $m=n+3,$ are labeled by $\mathcal{S}$, as in \eqref{eq:configurations-labels}. We consider the change of variables on $\mathcal{M}_{0,m}$:
$$
(z_1, \dots, z_n) \mapsto \left(x_1 = \frac{z_1}{z_2}, \dots, x_i=\frac{z_i}{z_{i+1}}, \dots, x_n = z_n\right), \hskip 0.3cm z_i=x_i \dots x_n.
$$
As in the Introduction, for any
$$a = (a_1, \dots, a_n) \in \Z_{\geq1}^n $$
we shall consider the polylogarithmic function (or a polylog)
$$L_a(x_1,\dots,x_n) = \sum_{0< k_1 < \dots <k_n} \, \frac{{x_1}^{k_1}\dots {x_n}^{k_n}}{k_1^{a_1} \dots k_n^{a_n}}$$
and more generally, if $\kappa\colon\{1,\dots, N\} \to \{1,\dots, n\}$ is a non-decreasing function and if $a\in\Z_{\geq1}^N$, the polylogarithmic function
$$L_{a,\kappa}(x_1,\dots,x_n) = L_a(y_{\kappa,1},\dots, y_{\kappa,N})$$
where
$$y_{\kappa,i} = z_{\kappa(i)} / z_{\kappa(i+1)} = \prod_{\kappa(i)\leq j <\kappa(i+1)} x_j \qquad(\text{with }\kappa(N+1) := \beta).$$
Note that we consider polylogs $L_{a, \kappa}(x)$ as functions of $(x_1, \dots, x_n)$. They are locally defined holomorphic functions on $\mathcal{M}_{0,m}$ (and extend to holomorphic functions on the universal cover of  $\mathcal{M}_{0,m}$). The open convergence domain of the corresponding power series is given by
\begin{equation}
\label{eq:domainD}
\mathcal{D} = \{ (x_1, \dots, x_n); |x_i| < 1 \,\, {\rm for} \,\, i=1, \dots, n\}.
\end{equation}
To $N=0$ we assign the polylogarithmic function $L_\emptyset(x_1,\dots,x_n)=1$.

\begin{Ex}
    For $n=1$, there is for each $N$ a unique map $\kappa$ given by the constant function with value $1$. In that case, polylogs are functions in one variable defined by the power series
    $$
L_{a,1}(x_1) = \sum_{0<k_1 < \dots < k_N} \frac{x_1^{k_N}}{k_1^{a_1} \dots k_N^{a_N}}.
    $$
    For $n=2$ and $\kappa$ given by $\kappa(j)=1$, $\kappa(j+1)=2$ for some $1\leq j\leq N$,
they take the form
    $$
L_{a, \kappa}(x_1, x_2)=\sum_{0<k_1 < \dots < k_N} \frac{x_1^{k_j}x_2^{k_N}}{k_1^{a_1} \dots k_N^{a_N}}
    $$
    For $\kappa(N) = 1$, we obtain polylog functions in one variable $L_{a,1}(x_1x_2)$, and for $\kappa(1) = 2$ we obtain  $L_{a,1}(x_2)$.
\end{Ex}

\begin{Prop}      \label{prop: dL}
De Rham differentials of polylogs $L_{a,\kappa}(x_1,\dots, x_n)$ are given by
\begin{align*}
d L_{a,\kappa}  &= \sum_{i=1, a_i\neq 1}^N L_{a_{(i)}^-, \kappa}\,\frac{dy_{\kappa,i}}{y_{\kappa,i}}\\
                &- \sum_{i=1, a_i = 1}^N L_{a_{(i)}^\wedge, \kappa\circ\delta_i}\,\frac{dy_{\kappa,i}}{y_{\kappa,i} - 1}\\
                &+ \sum_{i=1, a_i = 1}^{N-1} L_{a_{(i)}^\wedge, \kappa\circ\delta_{i+1}}\,\Bigl(\frac{dy_{\kappa,i}}{y_{\kappa,i} - 1} - \frac{dy_{\kappa,i}}{y_{\kappa,i}}\Bigr),
\end{align*}
where
\begin{align*}
a_{(i)}^- &= (a_1, \dots, a_i-1,\dots, a_N) \in \Z_{\geq 1}^{N}\\
a_{(i)}^\wedge &= (a_1, \dots, \widehat{a_i},\dots, a_N) \in \Z_{\geq 1}^{N-1}
\end{align*}
and the ``face maps'' $\delta_i\colon\{1,\dots, N-1\}\to\{1,\dots, N\}$
are given by
$$\delta_i(j) = 
\begin{cases}
    j & \text{if } j<i,\\
    j+1 & \text{otherwise}.
\end{cases}
$$
If $y_{\kappa, i}\equiv 1$, the 1-form $\frac{dy_{\kappa,i}}{y_{\kappa,i} - 1}$ is understood to be $0$.
\end{Prop}

\begin{proof}
The proof is by direct computation. Here we give a calculation for one of the interesting cases. Let $a=(1,1)$, $N=2$, $\kappa=\id$ (so, $y_{\kappa, i} = x_i$),
and
$$
L_{(1,1)}(x_1,x_2) = \sum_{0<k_1 < k_2} \frac{x_1^{k_1}x_2^{k_2}}{k_1k_2}.
$$
We compute,
\begin{align*}
d L_{(1,1)}(x_1,x_2) ={}&  \sum_{0<k_1 < k_2} \left( \frac{x_1^{k_1}x_2^{k_2}}{k_2}
\, \frac{dx_1}{x_1} +  \frac{x_1^{k_1}x_2^{k_2}}{k_1}
\, \frac{dx_2}{x_2}\right) \\
={}&   \sum_{k_2} \frac{x_1 - x_1^{k_2}}{1-x_1} \, \frac{x_2^{k_2}}{k_2}
\, \frac{dx_1}{x_1} + \sum_{k_1} \frac{x_2^{k_1}}{1-x_2} \, \frac{x_1^{k_1}}{k_1}
\, \frac{dx_2}{x_2} \\
={}&   L_{(1)}(x_2) \frac{dx_1}{1-x_1} - L_{(1)}(x_1x_2)\left(\frac{dx_1}{x_1} + 
\frac{dx_1}{1-x_1}\right)\\ &+ L_{(1)}(x_1x_2) \frac{dx_2}{1-x_2}\\
={}&  L_{(1),\delta_1}(x_1, x_2) \frac{dx_1}{1-x_1} - L_{(1),\delta_2}(x_1, x_2)\left(\frac{dx_1}{x_1} + 
\frac{dx_1}{1-x_1}\right)\\
&+ L_{(1), \delta_2}(x_1, x_2) \frac{dx_2}{1-x_2}
\end{align*}
as expected.
\end{proof}

\subsection{Iterated integrals and formal polylogs}

For a given base point on $\mathcal{M}_{0,m}$, the Chen iterated integral defines a map from  $\mathcal{V}(\mathcal{M}_{0,m})=(U\mathfrak{p}_m)^*$ to holomorphic  functions on the universal cover of $\mathcal{M}_{0,m}$. We denote by $\iota$ the iterated integral starting at the tangential base point located at 
$x_1= \dots =x_n=0$ with the tangent vector $(1, \dots, 1)$. In particular, for the counit $\epsilon\in(U\mathfrak{p}_m)^*$ we have $\iota(\epsilon) = 1$.

By Proposition $4.1$ of \cite{zhao_analytic_1999},  polylogs $L_{a,\kappa}(x)$ are in the image of $\iota$. To simplify the notation, we shall restrict these functions to the domain $\mathcal D$ defined by \eqref{eq:domainD}.
The inclusion $H^\bullet(\mathcal{M}_{0,m})\hookrightarrow \Omega^\bullet(\mathcal{M}_{0,m})$ via Arnold's 1-forms then gives us the map
$$
\hat{\iota}=(\iota \otimes 1): \mathcal{V}(\mathcal{M}_{0,m}) \otimes H^\bullet(\mathcal{M}_{0,m}) \to \Omega^{\rm hol}(\mathcal{D}),
$$
where $\Omega^{\rm hol}(\mathcal{D})$ stands for holomorphic differential forms on the domain $\mathcal{D}$. The information about the map $\hat{\iota}$ can be summarized in the following theorem (see \cite{chen_iterated_1977}, Theorem $4.2.1$ and \cite{brown_multiple_2009}, Section $3.6$):

\begin{Th}
    The map $\hat{\iota}$ is an injective map of differential graded commutative algebras, where on $\mathcal{V}(\mathcal{M}_{0,m})\otimes H^{\bullet}(\mathcal{M}_{0,m})$ one considers the differential \eqref{eq: d}, and on $\Omega^{\rm hol}(\mathcal{D})$ the de Rham differential.
\end{Th}

For any pair $(a,\kappa)$ we shall denote by $l_{a,\kappa}\in (U\mathfrak{p}_m)^*$ the $\iota$-preimage of the function $L_{a,\kappa}$. We then define, for a given $m = n + 3$, the vector subspace $\mathcal{L}_m \subset (U\mathfrak{p}_m)^*$ via
$$
%\begin{align*}
%\mathcal{L}^m_< &= \mathbb{C}\langle l_{a,\kappa} \mid N\geq 1,\;  a \in \mathbb{Z}_{\geq 1}^N,\; \kappa\colon\{1,\dots, N\}\to\{1,\dots, n\}\text{ increasing}\rangle\\
\mathcal{L}_m = \mathbb{C}\langle l_{a,\kappa} \mid N\geq 1,\,  a \in \mathbb{Z}_{\geq 1}^N,\; \kappa\colon\{1,\dots, N\}\to\{1,\dots, n\}\text{ non-decreasing}\rangle.
%\end{align*}
$$
We denote by 
\begin{equation*}      \label{eq:pi}
    \pi: U\mathfrak{p}_m \to \mathcal{L}_m^*
\end{equation*}
the natural projection dual to the injection $\mathcal{L}_m \rightarrowtail (U\mathfrak{p}_m)^*$.

\subsection{$U\mathfrak{p}_m$-module structure on $\mathcal{L}_m^*$}

%In this Section, we will define a structure of $U\mathfrak{p}_m$-modules on the duals $\mathcal{L}_m^*$  of the spaces $\mathcal{L}_m$ spanned by formal polylogs. 

Define a map
$$
\delta=\hat{\iota}^{-1} \circ d \circ \iota: 
\mathcal{L}_m \to (U\mathfrak{p}_m)^* 
\otimes H^1(\mathcal{M}_{0,m}),
$$
where $d$ is the de Rham differential, and $V=H^1(\mathcal{M}_{0,m})$ is represented by Arnold 1-forms.

\begin{Prop}
    The natural differential \eqref{eq: d} on $(U\mathfrak{p}_m)^* \otimes H(\mathcal{M}_{0,m})$ restricts to $\mathcal{L}_m \otimes H(\mathcal{M}_{0,m})$, and induces a $U\mathfrak{p}_m$-module structure on $\mathcal{L}_m^*$. The natural projection $\pi: U\mathfrak{p}_m \to \mathcal{L}_m^*$ is a $U\mathfrak{p}_m$-module morphism. The kernel ${\rm ker} \, \pi = \mathcal{L}_m^\perp$ is a left ideal in $U\mathfrak{p}_m$.
\end{Prop}

%\begin{Prop}
%    The map $\delta$ takes values in $\mathcal{L}_m \otimes H^1(\mathcal{M}_{0,m})$. When extended to $\mathcal{L}^m \otimes H(\mathcal{M}_{0,m})$, it squares to zero, and defines a $U\mathfrak{p}_m$-module structure on $\mathcal{L}_m^*$. 
%\end{Prop}

\begin{proof}
  By Proposition \ref{prop: dL}, de Rham differentials of polylogs are given by linear combinations of Arnold 1-forms with coefficients given by polylogs. By injectivity of $\hat{\iota}$, this implies that 
  the map $\delta$ takes values in $\mathcal{L}_m \otimes H^1(\mathcal{M}_{0,m})$, as required.

  Since the de Rham differential on $\Omega^{\bullet}(\mathcal{M}_{0,m})$ squares to zero, so does its preimage on $\mathcal{L}_m \otimes H^{\bullet}(\mathcal{M}_{0,m})$ given by the extension of $\delta$. By Proposition \ref{prop: key}, this implies that $\delta$ defines a
  $U\mathfrak{p}_m$-module structure on $\mathcal{L}_m^*$.  

  By Proposition \ref{prop:new}, the kernel of the projection $\pi$ is a left ideal in $U\mathfrak{p}_m$, and it coincides with $\mathcal{L}_m^\perp$.
\end{proof}

\begin{Cor}\label{cor: delta_l}
The coaction map $\delta$ takes the following values on formal polylogs $l_{a,\kappa}$:
\begin{align*}
\delta l_{a,\kappa}  &= \sum_{i=1, a_i\neq 1}^N l_{a_{(i)}^-, \kappa}\,\otimes\frac{dy_{\kappa,i}}{y_{\kappa,i}}\\
                &- \sum_{i=1, a_i = 1}^N l_{a_{(i)}^\wedge, \kappa\circ\delta_i}\,\otimes\frac{dy_{\kappa,i}}{y_{\kappa,i} - 1}\\
                &+ \sum_{i=1, a_i = 1}^{N-1} l_{a_{(i)}^\wedge, \kappa\circ\delta_{i+1}}\,\otimes\Bigl(\frac{dy_{\kappa,i}}{y_{\kappa,i} - 1} - \frac{dy_{\kappa,i}}{y_{\kappa,i}}\Bigr),
\end{align*}
where
\begin{align*}
a_{(i)}^- &= (a_1, \dots, a_i-1,\dots, a_N) \in \Z_{\geq 1}^{N}\\
a_{(i)}^\wedge &= (a_1, \dots, \widehat{a_i},\dots, a_N) \in \Z_{\geq 1}^{N-1}
\end{align*}
and the ``face maps'' $\delta_i\colon\{1,\dots, N-1\}\to\{1,\dots, N\}$
are given by
$$\delta_i(j) = 
\begin{cases}
    j & \text{if } j<i,\\
    j+1 & \text{otherwise}.
\end{cases}
$$
If $y_{\kappa, i}\equiv 1$, $\frac{dy_{\kappa,i}}{y_{\kappa,i} - 1}$ is understood to be $0$.
\end{Cor}

\begin{Prop}\label{prop:l-pairing}
For any monomial $w$ in $U\mathfrak{f}_{m-2}\subset U\mathfrak{p}_m$ and any $(a,\kappa)$ we have
$$
l_{a,\kappa}(w)=
\begin{cases}
(-1)^N & \text{if } w=w_{a,\kappa},\\
0 & \text{otherwise}.
\end{cases}
$$
\end{Prop}

\begin{proof}
The proof is by induction on the weight $|a|:=\sum_i a_i$. For $|a|=0$, i.e.\ $a=\emptyset$, we have $l_\emptyset = \varepsilon$ and $w_\emptyset=1$, so the statement is true.

By Corollary \ref{cor: delta_l}, we have
\begin{equation}\label{eq:proj-delta}
(1\otimes p_2)\circ \delta (l_{a,\kappa}) =
\begin{cases}
l_{a^-, \kappa} \otimes \frac{dx_n}{x_n} & \text{if }a_N > 1,\\
-l_{a', \kappa'}\otimes \frac{dz_{\kappa(N)}}{z_{\kappa(N) - 1}} & \text{if } a_N = 1,
\end{cases}
\end{equation}
where $a^-=(a_1,\dots,a_{N-1}, a_N-1)$,  $a'=(a_1,\dots,a_{N-1})$ and $\kappa' = \kappa|_{\{1,\dots,N-1\}}$. Since
$$(1\otimes p_2)(\Omega) = \sum_{i=1}^n X_{i \beta}\otimes \frac{dz_i}{z_i -1 } + X_{\beta \omega}\otimes \frac{dx_n}{x_n},$$
we get from \eqref{eq:proj-delta} for any $w'\in U\mathfrak{f}_{n+1}$ of weight $|a|-1$,
\begin{align*}
l_{a,\kappa}(X_{i\beta}\,w') &= 
\begin{cases}
    -l_{a',\kappa'}(w') & \text{if } a_N=1 \text{ and } i=\kappa(N),\\
    0 & \text{otherwise},
\end{cases}\\
l_{a,\kappa}(X_{\beta\omega}\,w') &= 
\begin{cases}
    l_{a^-, \kappa}(w') & \text{if }a_N>1,\\
    0 &  \text{if }a_N=1.
\end{cases}
\end{align*}
By assumption, the statement of the Proposition holds true for the right-hand sides, and therefore it is also true for $l_{a,\kappa}$.
\end{proof}

\begin{Prop}\label{prop:l-f-ideal}
The $U\mathfrak{f}_{m-2}$-module morphism $\pi|_{U\mathfrak{f}_{m-2}}\colon U\mathfrak{f}_{m-2}\to \mathcal{L}_m^*$ is surjective and its kernel is the left ideal $I_m$.
\end{Prop}
\begin{proof}
  By the previous Proposition \ref{prop:l-pairing}, the pairing between $\mathcal{L}_m \subset (U\mathfrak{p}_m)^*$ and $U\mathfrak{f}_{m-2} \subset U\mathfrak{p}_m$ has a trivial left kernel, which proves surjectivity. Its right kernel is the linear span of monomials which are not of the form $w_{a,\kappa}$. Hence, it is equal to $I_m$.
\end{proof}

\begin{Prop}\label{prop:s-in-kernel}
We have $U\mathfrak{p}_m\,s(\mathfrak{p}_{m-1})\subset\ker\pi$.

%For every $l_{a,\kappa}\in\mathcal L_m$ we have $l_{a,\kappa}(U\mathfrak{p}_m\,s(\mathfrak{p}_{m-1}))=\{0\}$, i.e.\ $U\mathfrak{p}_m\,s(\mathfrak{p}_{m-1})\subset\ker\pi$.
\end{Prop}
\begin{proof}
Since $\ker\pi$ is a left ideal in $U\mathfrak{p}_m$, it is enough to show that the map $\pi$ vanishes on generators of $s(\mathfrak{p}_{m-1})$. 
These generators are of the form
$X_{\alpha i}$ for $1 \leq i \leq n-1$ and $X_{ij}$ for $1 \leq i,j \leq n$.
For a formal polylog $l_{a,\kappa}$, we have
$$
l_{a,\kappa}(X) = \langle\delta l_{a,\kappa}, 1\otimes X\rangle,
$$
where we view $X$ as an element of $H_1(\mathcal M_{0,m})$ to make sense of the pairing. 

Observe that for any $(\alpha',\kappa')$, we have $l_{\alpha',\kappa'}(1)=0$ unless $\alpha'=\emptyset$. The expression for $\delta l_{a,\kappa}$ in Proposition \ref{prop: dL} shows that $(1\otimes p_2)(l_{a,\kappa})$ contains no $l_\emptyset$. Hence, we conclude that $\langle\delta l_{a,\kappa}, 1\otimes X\rangle=0$ for all generators $X$ of $s(\mathfrak{p}_{m-1})$ and for all formal polylogs $l_{a,\kappa}$. This completes the proof.
\end{proof}

We can now finally prove our main results.

\begin{proof}[Proofs of Theorems \ref{th: intro} and \ref{th: intro2}]
Theorem \ref{th: intro2} follows immediately from Propositions \ref{prop:l-pairing} and \ref{prop:s-in-kernel}. To show Theorem \ref{th: intro},  we consider the joint kernel of formal polylogs
$$
\ker\pi\subset U\mathfrak{p}_m = U\mathfrak{f}_{m-2}\oplus U\mathfrak{p}_m\,s(\mathfrak{p}_{m-1}).
$$
By Proposition \ref{prop:s-in-kernel}, it contains the space $U\mathfrak{p}_m\,s(\mathfrak{p}_{m-1})$, and by Proposition \ref{prop:l-f-ideal} we have $  U\mathfrak{f}_{m-2} \cap \ker\pi = I_m$. Hence,
$$
\ker\pi = I_m \oplus U\mathfrak{p}_m\,s(\mathfrak{p}_{m-1}) = J_m,
$$
as required.
\end{proof}

\section{Polylogs and free Lie algebras} \label{sec: polylogs and free}

Recall that the natural map $\pi: U\mathfrak{p}_m \to \mathcal{L}_m^*$ is surjective. The aim of this Section is to describe the spaces
$$
\mathfrak{pl}_m  := \pi(\mathfrak{p}_m) \subset \mathcal{L}_m^*
$$
and Lie subalgebras
    \begin{equation*}
        \mathfrak{k}_m := \mathfrak{p}_m \cap J_m
    \end{equation*}
%
% $$
% \mathfrak{k}_m:= \mathfrak{f}_{m-2} \cap I_m
% $$
%
%\begin{equation*}
%begin{split}
  %  \mathfrak{pl}_m & := \pi(\mathfrak{p}_m) \subset \mathcal{L}_m^*, \\
   % \mathfrak{k}_m & := \mathfrak{f}_{m-2} \cap I_m,
%\end{split}
%\end{equation*}
for $m=4$ and $m=5$. It will be enough to study
$$
\mathfrak{l}_m:= \mathfrak{f}_{m-2} \cap I_m,
$$
since $\mathfrak{k}_m = \mathfrak{l}_m \oplus s(\mathfrak{p}_{m-1})$.
Recall that $\mathfrak{p}_m = \mathfrak{f}_{m-2} \oplus s(\mathfrak{p}_{m-1})$ and $s(\mathfrak{p}_{m-1}) \subset {\rm ker}(\pi)$, hence 
\begin{equation}
\label{eq:plm}
\mathfrak{pl}_m = \pi(\mathfrak{f}_{m-2}) = 
\mathfrak{f}_{m-2}/\mathfrak{l}_m.
\end{equation}
%In what follows, we consider the cases of $m=4$ and $m=5$. 

\subsection{Computation of $\mathfrak{k}_4$ and $\mathfrak{pl}_4$}
In the case of $m=4$, the Lie algebra $\mathfrak{f}_2$ is spanned by the generators $X_{1 \beta}$ and $X_{\beta \omega}$, and the ideal $I_4$ is of the form
\begin{equation}
\label{eq:I4}
I_4 = (U\mathfrak{f}_2) X_{\beta \omega}.
\end{equation}

We show the following.
\begin{Prop}       \label{prop: m=4}
If $m=4$, then
    \begin{equation*}
            \mathfrak{k}_4 = \mathfrak{l}_4 =\mathbb{C}X_{\beta \omega} \quad \text{ and } \quad
            \mathfrak{pl}_4 \cong \mathfrak{f}({\rm ad}_{X_{\beta \omega}}^k X_{1 \beta}; k=0, 1, \dots).
    \end{equation*}
\end{Prop}

\begin{proof}
    Combining \eqref{eq:plm} and \eqref{eq:I4}, we get 
        \begin{equation*}
            \mathfrak{pl}_4 = \mathfrak{f}_2/(\mathfrak{f}_2 \cap (U\mathfrak{f}_2) X_{\beta \omega}).
        \end{equation*}
It follows from the PBW theorem that for a Lie algebra $\mathfrak{g}$ and a Lie subalgebra $\mathfrak{h} \subset \mathfrak{g},$ 
        \begin{equation*}
            \mathfrak{g} \cap (U\mathfrak{g}) \mathfrak{h} = \mathfrak{h},
        \end{equation*}
hence we conclude that,
\begin{equation*}
        \mathfrak{k}_4 =\mathbb{C} X_{\beta \omega}
\end{equation*}
and since $s(\mathfrak{p}_3)=0$,
\begin{equation*}
    \mathfrak{l}_4 = \mathfrak{k}_4.
\end{equation*}
Moreover,
\begin{equation*}
     \mathfrak{pl}_4 = \mathfrak{f}_2/\mathbb{C} X_{\beta \omega}
\cong \mathfrak{f}({\rm ad}_{X_{\beta \omega}}^k X_{1 \beta}; k=0, 1, \dots),
\end{equation*}
where $\mathfrak{f}({\rm ad}_{X_{\beta \omega}}^k X_{1 \beta}; k=0, 1, \dots)$ is a free Lie algebra with an infinite number of generators 
${\rm ad}_{X_{\beta \omega}}^k X_{1 \beta}$. The last equivalence follows from the Shirshov-Witt theorem; see \cite{reutenauer_free_2003}, Section $2.2$.
\end{proof}

\subsection{Computation of $\mathfrak{k}_5$ and $\mathfrak{pl}_5$}
In what follows, we use the shorthand notation $x = X_{1 \beta}, y = X_{2 \beta}, z= X_{\beta \omega}$ for the three generators of $\mathfrak{f}_3$. Denote by $[x]$ and $[y]$ the Lie ideals generated by $x$ and $y$, respectively, and by $\lambda :=\bigl[ [x], [y] \bigr]$ the Lie ideal spanned by Lie brackets of elements of $[x]$ and $[y]$. We will show:
\begin{Th}       \label{thm: m=5}
If $m=5$, then
\begin{equation*}
    \mathfrak{l}_5 =\mathbb{C}z \oplus [\lambda,\lambda], \quad \mathfrak{k}_5 =\mathbb{C}z \oplus [\lambda,\lambda] \oplus s(\mathfrak{p}_4)
\end{equation*}
and
\begin{equation*}
    \mathfrak{pl}_5 \cong f_2'(x,z) \oplus f_2'(y,z) \oplus \lamlam.
\end{equation*}
\end{Th}

In the case $m=5$, the Lie algebra $\mathfrak{f}_3$ is spanned by the generators $x,y$ and $z$, and the ideal $I_5 \subset U\mathfrak{f}_3$ takes the form $I_5 = (x)(y) + (U\mathfrak{f}_3)z$. Recall from Proposition \ref{prop:C-complement-I} that $C_5 \subset U\mathfrak{f}_3$ is a complement of $I_5$, % =(U\mathfrak{f}_3)z + (x)(y)$, 
that is, $C_5 \cong U\mathfrak{f}_3/I_5$.
It is useful to note that $\mathfrak{f}_3(x,y,z)$ admits a bigrading in the $x$- and $y$-degrees; it decomposes into the homogeneous parts of degrees $(0,0), (1,0), (0,1)$ and $\geq(1,1)$ as follows.
\begin{Prop}
\label{prop:decomp-f3}
\begin{equation*}
    \mathfrak{f}_3(x,y,z) \cong \mathbb{C}z \oplus \mathfrak{f}^{'}_2(x,z) \oplus \mathfrak{f}^{'}_2(y,z) \oplus \lambda,
\end{equation*}
where 
\begin{equation*}
\mathfrak{f}^{'}_2(x,z) \cong \mathfrak{f}(\ad^k_z(x); k\geq0), \quad
\mathfrak{f}^{'}_2(y,z) \cong \mathfrak{f}(\ad^k_z(y); k\geq0).
\end{equation*}
%and $\lambda :=\bigl[ [x], [y] \bigr]$ is the Lie ideal spanned by Lie brackets of elements of $[x]$ and $[y]$, where $[x]$ and $[y]$ are the Lie ideals spanned by $x$ and $y$, respectively.
\end{Prop}

Since the map $\pi$ is injective in the $(x,y)$-degrees $(1,0)$ and $(0,1)$, i.e. on $\mathfrak{f}^{'}_2(x,z)$ and $\mathfrak{f}^{'}_2(y,z)$, it remains to study the image of $\lambda$. Note that
\begin{equation*}
\begin{split}
[\lambda, \lambda] = \bigr[ \bigl[ [x], [y] \bigr], \bigl[ [x], [y] \bigr] \bigl] & \subset (x)(y), \\
\mathbb{C}z \oplus [\lambda, \lambda] & \subset \ker(\pi),
\end{split}
\end{equation*}
hence the map $\pi$ factors through 
\[ \begin{tikzcd}
\lambda \arrow{r}{\pi} \arrow[swap]{d} & (\mathcal{L}_5^{*})^{\geq (1,1)} \\%
\lambda/[\lambda,\lambda] \arrow[swap]{ur}{\varphi}
\end{tikzcd},
\]
% \[ \begin{tikzcd}
% \lambda \arrow[twoheadrightarrow]{r}{\pi} \arrow[swap]{d} & (\mathcal{L}_5^{*})^{\geq (1,1)} \\%
% \lambda/[\lambda,\lambda] \arrow[swap, twoheadrightarrow]{ur}{\varphi}
% \end{tikzcd},
% \]
% \begin{equation*}
%     \lambda/ [\lambda,\lambda] \to (\mathcal{L}_5^{*})^{\deg \geq (1,1)},
% \end{equation*}
where $(\mathcal{L}_5^{*})^{\geq (1,1)}$ denotes the graded subspace of $\mathcal{L}_5^{*}$ whose degrees in $x$ and $y$ are non-zero. 

\begin{Th}
\label{th:isomorphism}
The map $\varphi$ is an isomorphism.
\end{Th}

To prove this theorem, we shall first study the surjectivity of $\varphi$, then its injectivity.

\begin{Prop}
\label{prop:bases}
    The following sets are all bases of $(\mathcal{L}_5^{*})^{\geq (1,1)}$:
    \begin{enumerate}[label=(\roman*)]
        \item \label{prop:basis1} $ \mathbb{C} \langle z^{k_1} y \ldots z^{k_s} y z^{l_1} x \ldots z^{l_t} x \rangle^{\deg\geq(1,1)} $;
        \item \label{prop:basis2} $ \mathbb{C} \langle ad_z^{k_1} (y) \ldots ad_z^{k_s} (y) ad_z^{l_1} (x) \ldots  ad_z^{l_t} (x) \rangle$;
        \item \label{prop:basis3} $\mathbb{C} \langle [[\ldots[[ad_z^{k_1} (y),[ \ldots ,[ad_z^{k_{s-1}}(y), [ad_z^{k_s} (y), ad_z^{l_1} (x)] ]]],ad_z^{l_2}(x)],\ldots], ad_z^{l_t} (x) ] \rangle.$
        %$ \mathbb{C} \langle {\color{cyan}[}{\color{purple}[}\ldots{\color{orange}[}{\color{magenta}[}ad_z^{k_1} (y),{\color{green}[} \ldots ,{\color{blue} [}ad_z^{k_{s-1}}(y), [ad_z^{k_s} (y), ad_z^{l_1} (x)] {\color{blue}]}{\color{green}]}{\color{magenta}]},ad_z^{l_2}(x){\color{orange}]},\ldots{\color{purple}]}, ad_z^{l_t} (x) {\color{cyan}]} \rangle$.
    \end{enumerate}
\end{Prop}

\begin{proof} 
    \begin{enumerate}
        \item Clear from the form of the complement; see Proposition \ref{prop:C-complement-I}.
        \item Consider an element of \ref{prop:basis2}; it has the form 
            \begin{equation*}
               w= ad_z^{k_1} (y) \ldots ad_z^{k_s} (y) ad_z^{l_1} (x) \ldots  ad_z^{l_t} (x),
            \end{equation*}
        where are least one $k_i$ and $l_j$ are greater than $0$. Developing this expression, we get a sum of monomials in $x,y$ and $z$, to each of which we associate an exponent vector as follows. Reading the monomial from left to right, record the number of $z$'s which appear in front of each occurrence of $y$ or $x$ (we ignore the $z$'s appearing at the end of the monomial). We equip the set of exponent vectors with the lexicographical order, i.e. given two sequences $(a_1, \ldots, a_p)$ and $(b_1, \ldots, b_q)$, $(a_1, \ldots, a_p) > (b_1, \ldots, b_q)$ if and only if either $a_i > b_i$ for some index $i$ (and $a_j = b_j$ for all $j<i$), or $p>q$. It follows from this construction that the largest exponent vector associated to the element $w$ is 
            \begin{equation*}
                (k_1, \ldots, k_s,l_1, \ldots, l_t).
            \end{equation*}
        Note that the elements of the basis \ref{prop:basis1} can also be indexed by exponent vectors, defined in the same way; for instance, the element $z^{k_1}y \ldots z^{k_s} y z^{l_1} x \ldots z^{l_t} x$ corresponds to the vector $(k_1, \ldots, k_s,l_1, \ldots, l_t)$.
        Consider the map
            \begin{equation*}
            \begin{split}
                \phi : ad_z^{k_1} (y) \ldots ad_z^{k_s} (y) ad_z^{l_1} (x)\ldots ad_z^{l_t} (x) \mapsto z^{k_1} y \ldots z^{k_s} & y z^{l_1} x \ldots z^{l_t} x \hspace{1em} \\ & + \text{ smaller terms}.
            \end{split}
            \end{equation*}
            It translates to a transformation matrix between the bases \ref{prop:basis1} and \ref{prop:basis2}, which is lower triangular and has only $1$'s on its diagonal, so it is invertible.
       
        \item Consider an element of \ref{prop:basis2}; it has the form 
            \begin{equation*}
               w= ad_z^{k_1} (y) \ldots ad_z^{k_s} (y) ad_z^{l_1} (x) \ldots  ad_z^{l_t} (x).
            \end{equation*}
        We define a bracketing on the elements of \ref{prop:basis2} as follows. First, bracket together the last $y$-term and the first $x$-term, $[ad_z^{k_s}(y), ad_z^{l_1}(x)]$. Then, bracket $ad_z^{k_{s-1}}(y)$ with $[ad_z^{k_s}(y), ad_z^{l_1}(x)]$, and continue inductively towards the left until reaching $ad_z^{k_1}(y)$. Finally, bracket the resulting expression with
        %$[ad_z^{k_1}(y), \ldots, [ad_z^{k_s}(y), ad_z^{l_1}(x)]]$ with 
        $ad_z^{l_2}(x)$, and continue inductively towards the right until reaching $ad_z^{l_t}(x)$. This procedure completely determines a corresponding element of \ref{prop:basis3}, which thus has the form
            \begin{equation*}
                w = [[\ldots[[ad_z^{k_1} (y),[ \ldots ,[ad_z^{k_{s-1}}(y), [ad_z^{k_s} (y), ad_z^{l_1} (x)] ]]],ad_z^{l_2}(x)],\ldots], ad_z^{l_t} (x) ],
            \end{equation*}
        By construction, the expansion of $w$ is the sum of monomials,
            \begin{equation*}
                ad_z^{k_1} (y) \ldots ad_z^{k_s} (y) ad_z^{l_1} (x)\ldots ad_z^{l_t} (x)
            \end{equation*}
        together with additional terms in which at least one occurrence of $x$ appears to the left of one of $y$, i.e. elements of the ideal $I_5$. 
        It follows that modulo $I_5$, the families \ref{prop:basis2} and \ref{prop:basis3} generate the same space.
    \end{enumerate}
\end{proof}

Let us now compute the Poincaré series of the domain and codomain of $\varphi$, $\lambda/[\lambda,\lambda]$ and $(\mathcal{L}_5^{*})^{\geq (1,1)}$.
We will denote by $s_1, s_2$ and $t$ the variables that track the number of $x,y$ and $z$'s, respectively.

\begin{Lemma}
\label{lemma:poincare-L5}
    The Poincaré series of $(\mathcal{L}_5^{*})^{\geq (1,1)}$ is given by
    \begin{equation*}
        P(s_1,s_2,t)=\frac{s_1 s_2}{(1-s_1-t)(1-s_2-t)}.
    \end{equation*}
\end{Lemma}

\begin{proof}
     Recall from Proposition \ref{prop:bases} that an element of the basis of $(\mathcal{L}_5^{*})^{\geq (1,1)}$ has the form 
        \begin{equation*}
            z^{k_1} y \ldots z^{k_a} y z^{l_1} x \ldots z^{l_b} x,
        \end{equation*}
    with at least one occurrence of $x$ and $y$. 
    First, note that $z^ky$ is the only word containing a single $y$ and $k$ $z$'s; it has generating function 
        \begin{equation*}
            P(s_2,t) = \sum_{k\geq 0} s_2 t^k = \frac{s_2}{1-t}.
        \end{equation*}
    It follows that $z^{k_1} y \ldots z^{k_a} y$, with $a \geq1$ occurrences of $y$, has the generating function
        \begin{equation*}
            P(s_2,t) = \sum_{a\geq \bf{1}} \left( \frac{s_2}{1-t} \right)^{a} = \frac{s_2}{1-s_2-t}.
        \end{equation*}
    Finally, combining these calculations with the analogous ones for the variable $x$, we conclude that the Poincaré series of $(\mathcal{L}_5^{*})^{\geq (1,1)}$ is
        \begin{equation*}
            P(s_1,s_2,t)=\frac{s_1 s_2}{(1-s_1-t)(1-s_2-t)}.
        \end{equation*}

\end{proof}

To compute the Poincaré series of $\lambda/[\lambda,\lambda]$, we use some elementary homological algebra; for the necessary background, we refer to \cite{weibel_introduction_1994}. Consider the following short exact sequence of left $\lambda$-modules:
\begin{equation*}
    0 \to \mathfrak{I} \to U\mathfrak{f}_3 \overset{\varepsilon}{\to} k \to 0,
\end{equation*}
where $\varepsilon : U\mathfrak{f}_3\to k$ is the \textit{augmentation} and %, i.e. the unique $k$-algebra homeomorphism sending $i(\mathfrak{f}_3)$ to zero,
$\mathfrak{J}$ is the \textit{augmentation ideal}. %, the two-sided ideal of $U\mathfrak{f}_3$ generated by $i(\mathfrak{f}_3)$ as a left ideal.
This short exact sequence induces a long exact sequence of Lie algebra homology (see \cite{weibel_introduction_1994}, Theorem 1.3.1):
\begin{equation}
\label{eq:long-exact-seq}
\begin{split}
    \ldots \to H_2(\lambda,k) \to H_1(\lambda,\mathfrak{J}) \to & H_1(\lambda,U\mathfrak{f}_3) \to H_1(\lambda,k) \\
    \to & H_0(\lambda,\mathfrak{J}) \to H_0(\lambda,U\mathfrak{f}_3) \to H_0(\lambda,k) \to 0.
\end{split}
\end{equation}
Since $\uf$ is a free $\lambda$-module, $H_i(\lambda,\uf)=0$ for $i\geq 1$, and \eqref{eq:long-exact-seq} restricts to the following short exact sequence:
\begin{equation*}
    0 \to H_1(\lambda,k) \to H_0(\lambda,\mathfrak{J}) \to H_0(\lambda,U\mathfrak{f}_3) \to H_0(\lambda,k) \to 0,
\end{equation*}
i.e.
\begin{equation}
\label{eq:short-seq-homlgy}
    0 \to \lambda/[\lambda,\lambda] \to U(f_3/\lambda)\otimes (kx \oplus ky \oplus kz) \to  U(\mathfrak{f}_3/\lambda)\to k \to 0.
\end{equation}

% We use:
% \begin{itemize}
%     \item $H_1(\lambda,k)=\lambda/[\lambda,\lambda]$
%     \item $H_0(\lambda,k)=k$
%     \item $H_0(\lambda,\mathfrak{I})=U(f_3/\lambda)\otimes (kx \oplus ky \oplus kz)$
%     \item $H_0(\lambda, \uf) = U(\mathfrak{f}_3/\lambda)$.
% \end{itemize}

\begin{Lemma}
\label{lemma:poincare-lambda}
    The Poincaré series of $\lambda/[\lambda,\lambda]$ is given by
    \begin{equation*}
        P(s_1,s_2,t)=\frac{s_1 s_2}{(1-s_1-t)(1-s_2-t)}.
    \end{equation*}
\end{Lemma}
\begin{proof}\footnote{We thank B. Enriquez and H. Furusho for explaining to us an alternative proof of this statement based on Lazard's elimination.}
Denote by $Q(s_1,s_2,t)$ and $R(s_1,s_2,t)$ the Poincaré series of $U(\mathfrak{f}_3/\lambda)$ and $U(f_3/\lambda)\otimes (kx \oplus ky \oplus kz)$ respectively. It follows from the exactness of the short sequence \eqref{eq:short-seq-homlgy} that 
    % \begin{equation*}
    %     \dim(\lambda/[\lambda,\lambda]) = \dim(U(\mathfrak{f}_3/\lambda)) - \dim \big(U(f_3/\lambda)\otimes (kx \oplus ky \oplus kz)\big). 
    % \end{equation*}
\begin{equation*}
    P(s_1,s_2,t) = R(s_1,s_2,t)-Q(s_1,s_2,t)+1.
\end{equation*}

It follows from Proposition \ref{prop:decomp-f3} that $\mathfrak{f}_3(x,y,z)/\lambda=\mathfrak{f}_2(x,z)\oplus_{\mathfrak{f}_1(z)}\mathfrak{f}_2(y,z)$. By a straightforward counting argument, we get that the Poincaré series of $\mathfrak{f}_2(x,z)$ is $\frac{1}{1-(s_1+t)}$, and that of $\mathfrak{f}_2(y,z)$ is  $\frac{1}{1-(s_2+t)}$. Hence,
\begin{equation*}
    Q(s_1,s_2,t)=\frac{1-t}{(1-s_1-t)(1-s_2-t)}.
\end{equation*}
Moreover, the term $kx \oplus ky \oplus kz$ has Poincaré series $s_1+s_2+t$, so,
\begin{equation*}
    R(s_1,s_2,t)=\frac{1-t}{(1-s_1-t)(1-s_2-t)}(s_1+s_2+t).
\end{equation*}
Finally, combining these two computations, we conclude
\begin{equation*}
\begin{split}
    P(s_1, s_2, t) & =\frac{1-t}{(1-s_1-t)(1-s_2-t)}(s_1+s_2+t-1)+1\\
    & = \frac{s_1 s_2}{(1-s_1-t)(1-s_2-t)}.
\end{split}
\end{equation*}
\end{proof}

\begin{proof}[Proof of Theorem \ref{th:isomorphism}]
    The surjectivity of $\varphi$ follows from Proposition \ref{prop:bases}, and the injectivity from the fact that the Poincaré series of the domain and codomain coincide, by Lemmas \ref{lemma:poincare-L5} and \ref{lemma:poincare-lambda}.
\end{proof}

We can now finally finish the proof of Theorem \ref{th: intro3}.

\begin{proof}[Proof of Theorem \ref{th: intro3}]
Combining the statements of Proposition \ref{prop: m=4} and of Theorem \ref{thm: m=5}, we obtain the statement of Theorem  \ref{th: intro3}.  
\end{proof}

% %
% $$
% I = (U \mathfrak{f}_3) X_{\beta \omega} + (X_{1 \beta})(X_{2 \beta}).
% $$
%  In this Section, we will use the shorthand notation
% for generators of $\mathfrak{f}_3$:
% %
% $$
% x = X_{1 \beta}, \hskip 0.3cm
% y = X_{2, \beta}, \hskip 0.3cm
% z= X_{\beta \omega}.
% $$

The Lie ideal $\lambda \subset \mathfrak{f}_3$ admits the following presentation in terms of the projections $\pi_i$:
\begin{Prop}       \label{prop:3cap}
    We have,
    $$
\lambda = {\rm ker}(\pi_1) \cap {\rm ker}(\pi_2) \cap {\rm ker}(\pi_\beta).
    $$
\end{Prop}

\begin{proof}
    Recall that ${\rm ker}(\pi_\beta) = \mathfrak{f}_3(X_{1 \beta}, X_{2 \beta}, X_{\beta \omega})$. Then,
    $$
{\rm ker}(\pi_1) \cap {\rm ker}(\pi_\beta) = [X_{1\beta}]
\quad \text{and} \quad
{\rm ker}(\pi_2) \cap {\rm ker}(\pi_\beta) = [X_{2\beta}],
    $$
    where $[X_{1\beta}]$ and $[X_{2\beta}]$ are Lie ideals in $\mathfrak{f}_3(X_{1 \beta}, X_{2 \beta}, X_{\beta \omega})$ generated by $X_{1\beta}$ and $X_{2\beta}$, respectively.
    Finally,
    $$
{\rm ker}(\pi_1) \cap {\rm ker}(\pi_2) \cap {\rm ker}(\pi_\beta)
= [X_{1\beta}] \cap [X_{2\beta}] = [[X_{1\beta}],[X_{2\beta}]]=
\lambda,
    $$
    as required.
\end{proof}

\begin{Rem}
    Let $\psi \in \mathfrak{f}_2$ be an element of degree at least 3.
    Denote by ${\rm pent}(\psi) \in \mathfrak{p}_5$ the left hand side of the infinitesimal pentagon equation for $\psi$ (for more details, see \cite{Drinfeld}):
    $$
    \begin{array}{lll}
{\rm pent}(\psi) & = & \psi(X_{\alpha1}, X_{12}) + 
+ \psi(X_{\alpha 1} + X_{\alpha 2}, X_{1\beta} + X_{2\beta})
+ \psi(X_{12}, X_{2\beta}) \\
& - & \psi(X_{\alpha 2} + X_{12}, X_{2\beta} )
- \psi(X_{\alpha 1}, X_{1\beta} + X_{12}).
\end{array}
    $$
    The expression ${\rm pent}(\psi)$ belongs to the joint kernel of projections $\pi_1, \pi_2$ and $\pi_\beta$. Hence, by Proposition \ref{prop:3cap}, it belongs to 
    the Lie ideal $\lambda$. B. Enriquez and H. Furusho suggested that 
     the condition 
    \begin{equation} \label{eq  : pent commutator}
    {\rm pent}(\psi) \in [\lambda, \lambda]
    \end{equation}
    is equivalent to $\psi \in \mathfrak{dmr}_0$.
    Note that by Theorem \ref{thm: m=5} equation \eqref{eq  : pent commutator} is equivalent to the condition
    \begin{equation}      \label{eq: pent pi}
    {\rm pent}(\psi) \in {\rm ker}(\pi)
    \end{equation}
stating that ${\rm pent}(\psi)$ belongs to the joint kernel of formal polylogs.
    \end{Rem}

\begin{Rem}      \label{rem:MeganMuze}
Let $\psi \in \mathfrak{f}_2$ be an element of degree at least $3$
verifying the skew-symmetric relation $\psi(x,y)=-\psi(y,x)$. 
Then, Theorems $3.8$ and $3.15$ in \cite{howarth_ren} imply  the 
equivalence of the condition 
\eqref{eq: pent pi} to $\psi \in \mathfrak{dmr}_0$.
\end{Rem}

\begin{Rem}
    N. Markarian has kindly informed us that he proved the following statement. Let $p_i\colon \pi_1^{un}(\mathcal M_{0,5}) \to \pi_1^{un}(\mathcal M_{0,4})$ be the map forgetting one points, $\Lambda=\cap_{i=1,2,3}\ker p_i$ the joint kernel of three projections, and  $f~\in~\pi_1^{un}(\mathcal{M}_{0,4})$ be symmetric element (that is, $f(x,y)=f^{-1}(y,x)$). In this case, ${\rm Pent}(f)\in \Lambda$, and coefficients of $f$ obey regularized double shuffle relations if and only if ${\rm Pent}(f)^{ab}\in \Lambda^{ab}$ vanishes. This statement is mentioned in the slides of Markarian's talk \cite{LectureA} (also with the reference to Enriquez-Furusho), and it will be explained in his forthcoming paper ``Multiplicative convolution and double shuffle relations''. The statement of Remark \ref{rem:MeganMuze} resembles of its infinitesimal version. 
\end{Rem}

\bibliographystyle{abbrv}
\bibliography{bibliography}

\end{document}